\documentclass[a4wide, 11pt, reqno]{amsart}
\usepackage{graphicx,subfigure,amsmath,amscd,amsfonts,mathrsfs,amssymb,amscd}
\graphicspath{{plots/}}
\title{Exceptional p-groups of order p$^5$}
\author{Sichao Jiang}
\date{\today}
\numberwithin{equation}{section}

\newtheorem{theorem}{Theorem}
\newtheorem{proposition}{Proposition}

\begin{document}
\begin{abstract}
The mininal degree of a finite group $G$, $\mu(G)$, is defined to be the smallest natural number $n$ such that $G$ embeds inside Sym$(n)$. The group $G$ is said to be exceptional if there exists a normal subgroup $N$ such that $\mu(G/N)>\mu(G)$. We will investigate the smallest exceptional $p$-groups, when $p$ is an odd prime. In ~\cite{Lemiuex's thesis} Lemiuex showed that there are no exceptional $p$-groups of order strictly less than $p^5$ and imposed severe restrictions on the existence of exceptional groups of order $p^5$. In fact he showed that if any were to exist, they must come from central extensions of four isomorphism classes of groups of order $p^4$. Then in ~\cite{Lemiuex} he exhibited an example of an exceptional group of order $p^5$. The author demonstrates the existence of two more exceptional groups arising in such a fashion and rules out the possibility of the remaining case. 
\end{abstract}

\maketitle

%\section{massive test}
%	cup .~\cite{latex} .~\cite{ConcreteMath}

\section{Introduction}
%\subsection{Elementary Permutation Theory} %%%how to get a newline?!?!?! // doesn't work%%%%%%%%%%%%%%%%
%This subsection will motivate the study of exceptional groups. In particular, we will discuss how one should consider permutation representations of groups. We will state without proof some of the basics of permutation theory. The proofs can be found for instance in (QUOTE NEILS PHD THESIS!). For the purposes of this paper, all groups are finite. %%(find a reference!!!)%%% 
%We shall start with some definitions.
%\begin{definition}
%Let $G$ be a group and $X$ a set. Denote the symmetric group on $X$ by $Sym(X)$. A \textbf{permutation representation} of $G$ is a group homomorphism $$\phi:G\longrightarrow Sym(X).$$ We call $\phi$ \textbf{faithful} if it is injective and \textbf{transitive} if for all $x,y\in X,$ there exists $g\in G$ such that $g\phi:x\mapsto y.$ 
%\end{definition}
%\begin{proposition}
%Every permutation representation of a group $G$ is a disjoint union of transitive permutation representations.
%\end{proposition}
%\begin{remark}
%In light of this proposition, one simply needs to consider transitive permutation representations. 
%\end{remark}
%\begin{definition}
For the purposes of this paper, all groups will be assumed to be finite. The \emph{minimal degree} $\mu(G)$ of a group $G,$ is the least non-negative integer $n$ such that $G$ embeds inside $Sym(n).$ That is, it is the smallest possible faithful permutation representation of $G.$ Define two permutation representations of $G$ by $\phi:G\longrightarrow Sym(X)$ and $\psi:G\longrightarrow Sym(Y).$ We say that $\phi$ and $\psi$ are \emph{equivalent} if there exists a bijection $\theta:X\longrightarrow Y$ such that the following diagram communtes:
\begin{center}
$\begin{CD}
X @>g\phi>> X\\
@VV\theta V @VV\theta V\\
Y @>g\psi>> Y
\end{CD}$
\end{center}
which is equivalent to saying that 
$$[x(g\phi)]\theta=(x\theta)g\psi,\quad \mathrm{for\  all}\  x\in X.$$
Now let $H$ be a subgroup of $G.$ The \emph{core} of $H$ in $G$, denoted core$H,$ is the largest normal subgroup of $G$ contained in $H.$ In particular if $H\lhd G,$ then $\text{core}H=H.$ \\
%\end{definition}

%\begin{definition}
%Let $G$ be a group. Define two permutation representations of $G$ by $\phi:G\longrightarrow Sym(X)$ and $\psi:G\longrightarrow Sym(Y).$ We say that $\phi$ and $\psi$ are %\textbf{equivalent} if there exists a bijection $\theta:X\longrightarrow Y$ such that the following diagram communtes:
%\begin{center}
%$\begin{CD}
%X @>g\phi>> X\\
%@VV\theta V @VV\theta V\\
%Y @>g\psi>> Y
%\end{CD}$
%\end{center}
%which is equivalent to saying that 
%$$[x(g\phi)]\theta=(x\theta)g\psi,\quad \mathrm{for\  all}\  x\in X.$$
%\end{definition}
%\begin{definition}
%Let $G$ be a group and $H$ a subgroup of $G$. The \textbf{core} of $H,$ core(H) is the largest normal subgroup of $G$ contained in $H$, or equivalently $$\mathrm{core}(H):=\bigcap_{g\in G}H^g.$$
%\end{definition}
Now we will make the most important construction of this section. We will define a permutation representation of $G$ by the group $Sym(G/H)$ for some subgroup $H\le G$ and show that in fact every transitive permutation representation of $G$ is equivalent to such a representation. Let $H$ be a subgroup of $G$. Let $G/H=\{Hx\ |\ x\in G\}$ denote the coset space. Define $$\phi_{H}:G\longrightarrow Sym(G/H)$$ by the rule $$g\phi_{H}:Hh\mapsto Hhg,$$ for all $h\in G$. 
%\begin{theorem}
Let $\phi:G\longrightarrow Sym(X)$ be a transitive permutation representation of $G$. Then $\phi$ is equivalent to $\phi_{H}$ for some subgroup $H$ of $G$. 
%\end{theorem}
%\begin{proof}
%Let $x_0\in X.$ Set $H=\{g\in G\ |\ x_0(g\phi)=x_0\},$ which is clearly a subgroup of $G$. Define $\theta:X\longrightarrow G/H$ as follows: $$x\theta=Hg,$$ where $x\in X$ is such that $g\phi:x_0\mapsto x$ for some $g\in G$. It is now routine to verify that $\theta$ is a bijection. 
%\end{proof}
This identification now gives us a more practical way of computing minimal faithful permuation representations of groups. One needs to find a collection of subgroups $H_1,\ldots ,H_n$ of $G$ such that the intersection of the cores of the ${H_i}'s$ is trivial. The minimal degree is then simply given by $$\mu(G)=\text{min}\bigg\{ \ \sum_{i=1}^n|G:H_i| \quad \bigg| \quad
\bigcap_{i=1}^n \text{core}(H_i)=\{1\}\bigg\}.$$ From now on, we shall think of permutation representations of $G$ as a disjoint union of $Sym(G/H_i)$ where $i=1,2,\ldots ,n.$\\
%\begin{definition}
\\
\section{Exceptional groups}
Now we define the key notion of an exceptional group. A group $G$ is said to be \emph{exceptional} if it has a normal subgroup $N$ with quotient $G/N$ such that $\mu(G)<\mu(G/N)$. In this case, $N$ is a \emph{distinguished subgroup} with $G/N$ the \emph{distinguished quotient}.
%\end{definition}
%\begin{remark}
At a first glance, one may think such a definition is silly since if $H$ is a subgroup of $G$, then necessarily $\mu(H)\le\mu(G).$ However, by analogy to the example in ~\cite{Neumann}, the author found an instance where this fails in spectacular fashion. Let $G=Q_8\times Q_8\times\ldots\times Q_8={Q_8}^n$ be the product of $n$ copies of the quaternion group and let
$$N=\langle (z,z,1,\ldots ,1), (1,z,z,1,\ldots,1),\ldots , (1,\ldots, z,z)\rangle,$$
where $z\neq 1$ is central. We have $\mu(G)=n\mu(D_8)=4n.$ But then $\mu(G/N)=2^{n+1}.$ This example shows that not only is the minimal degree of the distinguished quotient potentially larger than the minimal degree of the group, but in some cases it increases exponentially with the size of the group. 
%\end{remark}
Here is a result due to Easdown and Praeger (1988) concerning the smallest exceptional groups:
\begin{proposition}
There are no exceptional groups of order strictly less than 32. The two exceptional groups of order $32$ are:
$$G=\langle x,y\ |\ x^8=y^4=1,\ x^y=x^{-1}\rangle$$\\
with minimal degree $12,$ with distinguished quotient\\
$$G/{\langle x^4y^2\rangle}\cong\langle x,y\ |\ x^8=y^4=1,\ y^2=x^4,\ x^y=x^{-1}\rangle$$\\
of minimal degree $16$ and \\
$$H=\langle x,y,n\ |\ x^8=y^4=n^2=1,\ y^2=x^4,\ x^y=x^{-1}n,\ n^x=n^y=n\rangle$$\\
with minimal degree $12,$ also with distinguished quotient\\
$$H/{\langle n\rangle}\cong\langle x,y\ |\ x^8=y^4=1,\ y^2=x^4,\ x^y=x^{-1}\rangle$$\\
of minimal degree $16$. 
\end{proposition}
\begin{proof}
See ~\cite{Easdown & Praeger}. 
\end{proof}
This example motivated the study of how small exceptional groups could actually be for a general prime $p$. This was what Lemiuex later worked on. 

\section{Theory of Johnson, Wright and Lemieux}
For the remainder of this paper, $p$ will denote an odd prime. The minimal degree problem for groups was first investigated by Johnson in ~\cite{Johnson}. He set the foundations for the this area of study and established some major theorems. A few years later, Wright, in ~\cite{Wright}, found a neat formula for the minimal degrees of direct products $p$-groups. In ~\cite{Lemiuex's thesis} Lemieux showed that there are no exceptional $p$-groups of order strictly less than $p^5$. Then in ~\cite{Lemiuex}, he demonstrated the existence of an exceptional group of order $p^5$. This section addresses without proof, (references to the relevant sourse are given) the contributions of the three authors. The following theorems are due to Johnson.  
\begin{theorem}
Let $G$ be a non-trivial $p$-group whose centre $Z$ is minimally generated by $d$ elements, and let $\mathscr{R}=\{G_1,\ldots,G_n\}$ be a minimal (faithful) repesentation of $G$. Then for $p$ an odd prime, we have $n=d$.
\end{theorem} 
\begin{proof}
See ~\cite{Johnson}
\end{proof}
\begin{theorem}\label{Johnson's bound}
Let $G$ be a non-abelian $p$-group which does not decompose as a non-trivial direct product and whose centre is either cyclic or elementary abelian. Then we have the following bound:
$$p\mu(Z)\le\mu(G)\le\frac{1}{p}|G:Z|\mu(Z).$$ 
\end{theorem}
\begin{proof}
See ~\cite{Johnson}
\end{proof}
The next theorem is Wright's main result.
\begin{theorem}
Let $G$ and $H$ be $p$-groups. Then $\mu(G\times H)=\mu(G)+\mu(H).$
\end{theorem}
\begin{proof}
Wright actually proved this for all nilpotent groups and for all primes $p,$ see ~\cite{Wright}.
\end{proof}
Lemieux in ~\cite{Lemiuex's thesis} studied small exceptional $p$-groups. Although he could not find the smallest exceptional $p$-groups, he severely limited the existence of such groups in the following theorem:
\begin{theorem}
There are no exceptional $p$-groups of order strictly less than $p^5$. If there are any exceptional $p$-groups of order $p^5$, then they must be a central extension of these 4 isomorphism classes of distinguished quotients which are $p$-groups of order $p^4:$ 
\begin{align*}
&G_1=\langle x,y,z\; |\; x^{p^2}=y^p=z^p=1,\; [x,y]=1,\; [x,z]=1, \; [y,z]=x^p\rangle \\
&G_2=\langle x,y,z\; |\; x^{p^2}=y^p=z^{p^2}=1,\; z^p=x^p,\; [x,y]=x^p,\; [x,z]=y, \; [y,z]=1\rangle \\
&G_3=\langle x,y,z\; |\; x^{p^2}=y^p=z^{p^2}=1,\; z^p=x^{\alpha p},\; [x,y]=x^p,\; [x,z]=y,\; [y,z]=1\rangle \\
&G_4=\langle x,y,z,w\; |\; x^p=y^p=z^p=w^p=1,\; [x,y]=z,\; [x,z]=[x,w]=[y,z]= \\
&\qquad\ \; [y,w]=[z,w]=1\rangle \\
&\text{where $\alpha$ is a quadratic non-residue mod $p$.}
\end{align*}
\end{theorem}
In ~\cite{Lemiuex} Lemiuex showed the existence of an exceptional group of order $p^5$ with distinguished quotient isomorphic to $G_1.$ The remaining cases will be examined in the next section. 

\section{The remaining cases} 
\begin{theorem}
Here are three exceptional $p$-groups of order $p^5$ and minimal degree $2p^2:$
\begin{align*}
&E_1=\langle x,y,z,n\ |\ x^{p^2}=y^p=z^p=n^p=1,\ [y,z]=x^pn^{-1},\ [x,z]=n,\\
&\qquad\ \; [x,y]=[x,n]=[y,n]=[z,n]=1\rangle \\
&E_2=\langle x,y,z,n\ |\ x^{p^2}=y^p=z^{p^2}=n^p=1,\ x^p=z^pn,\ [x,y]=x^p, \\ 
&\qquad\ \; [x,z]=y,\ [y,z]=[x,n]=[y,n]=[z,n]=1\rangle \\
&E_3=\langle x,y,z,n\ |\ x^{p^2}=y^p=z^{p^2}=n^p=1,\, z^p=x^{\alpha p}n,\, [x,z]=yn,\\
&\qquad\ \; [x,y]=x^p,\, [y,z]=[x,n]=[y,n]=[z,n]=1\rangle \\
\end{align*}
with distinguished quotients isomorphic to $G_1$, $G_2$ and $G_3$ respectively.  %%%%%%%%%%%%%%%%%%better to spell G_i's out in full or???%%%%%%%%% 
\end{theorem}
\begin{proof}
That $E_1$ is exceptional with distinguished quotient $G_1$ was shown in ~\cite{Lemiuex}. We now show that $E_2$ is exceptional with distinguished quotient isomorphic to $G_2.$ It is clear that $E_2/{\langle n\rangle}\cong G_2.$ The centre of $E_2$ is $\langle x^p,n\rangle,$ and so by Theorem 1, any minimal faithful permutation representation of $E_2$ must have two orbits. Let 
$$H_1=\langle x,y\rangle\cong\langle x,y\ |\ x^{p^2}=y^p=1,\ [x,y]=x^p\rangle\qquad\text{and}$$
$$H_2=\langle y,z\rangle\cong\mathbb{Z}_{p^2}\times\mathbb{Z}_p$$
be subgroups of $E_2$ of order $p^3.$ Observe that $H_1\lhd E_2$ and $H_2\ntriangleleft E_2,$ so that core$H_1=H_1$ and core$H_2\neq H_2.$ Therefore $|\text{core}H_2|=1,p$ or $p^2.$ If $|\text{core}H_2|=1,$ then $\text{core}H_1\cap\text{core}H_2=\{1\}$ and we are done. Now suppose $|\text{core}H_2|=p^2$ so that $\text{core}H_2\cong\mathbb{Z}_{p^2}\ \text{or}\ \mathbb{Z}_{p}\times\mathbb{Z}_{p}.$ Thus without loss of generality, $\text{core}H_2=\langle z\rangle\ \text{or}\ \langle zy\rangle\ \text{or}\ \langle z^p,y\rangle.$ However none of these are normal subgroups of $E_2$ since $[x,z]=y,$ $(z^{-1}y^{-1})^x=z^{p-1}n$ and $(y^{-1})^x=z^py^{-1}n$ respectively. Hence $|\text{core}H_2|=p,$  so that without loss of generality, $\text{core}H_2\cong\langle y\rangle\ \text{or}\ \langle z^p\rangle\ \text{or}\ \langle z^py\rangle.$ However of the three, only $\langle z^p\rangle\lhd E_2.$ Hence $\text{core}H_2=\langle z^p\rangle=\langle x^pn^{-1}\rangle.$ Whence we have 
$$\text{core}H_1\cap\text{core}H_2=H_1\cap\langle x^pn^{-1}\rangle=\{1\}.$$ Now we get $$2p^2=p\mu(Z(E_2))\le\mu(E_2)\le|E_2:H_1|+|E_2:H_2|=p^2+p^2=2p^2<p^3=\mu(G_2)$$ 
and so $\mu(E_2)=2p^2.$ The first inequality is by Theorem 2 and the last equality is in the content of table $1$ in ~\cite{Lemiuex}.\\

Now we proceed to show that $E_3$ is exceptional with distinguished quotient $G_3$. Firstly, it is obvious that $E_3/{\langle n\rangle}\cong G_3.$ Observe that the centre of $E_3$ is $\langle x^p,n\rangle,$ and so by Theorem 1, any minimal faithful permutation representation of $E_3$ must have two orbits. Let 
$$S_1=\langle x,y\rangle\cong\langle x,y\ |\ x^{p^2}=y^p=1,\ [x,y]=x^p\rangle\qquad\text{and}$$
$$S_2=\langle y,z\rangle\cong\mathbb{Z}_{p^2}\times\mathbb{Z}_p$$ 
be the two subgroups of $E_3$ corresponding to the two orbits. We need to show that $\text{core}S_1\cap\text{core}S_2=\{1\}.$ Note that $S_1\ntriangleleft E_3.$ Now observe that $\langle x^p,y\rangle$ is a maximal subgroup of $S_1$, of order $p^2$ and normal in $E_3.$ Thus $\text{core}S_1=\langle x^p,y\rangle.$ It is easy to see that $S_2\ntriangleleft E_3.$ If $|\text{core}S_2|=p^2,$ then without loss of generality we may assume $\text{core}S_2\cong\langle z\rangle\ \text{or}\ \langle y,z^p\rangle.$ Now neither $\langle z\rangle$ nor $\langle y,z^p\rangle$ are normal in $E_3.$ Thus $|\text{core}S_2|=p$ and we conclude without loss of generality that $\text{core}S_2\cong\langle z^p\rangle.$ Now
$$\text{core}S_1\cap\text{core}S_2=\langle x^p,y\rangle\cap\langle z^p\rangle=\langle z^pn,y\rangle\cap\langle z^p\rangle=\{1\}.$$ So we have 
$$2p^2=p\mu(Z(E_3))\le\mu(E_3)\le|E_2:S_1|+|E_2:S_2|=p^2+p^2=2p^2<p^3=\mu(G_3),$$ and so $\mu(E_3)=2p^2$ and we are done.
\end{proof}
Note that the group $G_4$ was omitted from the considerations of the theorem. Now we will prove that $G_4$ does not centrally extend to an exceptional group of order $p^5.$ 
\begin{theorem}
Centrally extending the group $G_4$ will not result in an exceptional $p$-group $E_4$ of order $p^5,$ with distinguished quotient isomorphic to $G_4$.  
\end{theorem}
\begin{proof}
%Centrally extending $G_4$ will result in a group of order $p^5$ with presentation given by: 
%\begin{align*}
%&E_4=\langle x,y,z,w\; |\; x^p=y^p=z^p=w^p=n^p=1,\; [x,y]=zn^{a_1},\; [x,z]=n^{a_2}\\
%&\qquad\ \;  [x,w]=n^{a_3},\; [y,z]=n^{a_4},\; [y,w]=n^{a_5},\; [z,w]=n^{a_6}\\
%&\qquad\ \;  [n,x]=[n,y]=[n,z]=[n,w]=1\rangle, \\
%\end{align*}
%where $0\le a_i\le p-1$ for $i=1,2,\ldots ,6$. (Note that we may take $a_1=0$, since varying $a_1$ results in a family of isomorphic groups by sending $z$ to $zn^{a_1}$.)  
Let $E_4$ be an extension of $G_4$ satisfying the conditions of the theorem. First we show that $Z(E_4)\cong\langle n\rangle,$ where $E_4/\langle n\rangle\cong G_4.$ If $p^2\le|Z(E_4)|,$ then by Theorem \ref{Johnson's bound}, we have $\mu(E_4)\ge 2p^2>p^2+p=\mu(G_4),$ contradicting that $E_4$ is exceptional. Hence we must have $|Z(E_4)|=p$ and thus $Z(E_4)=\langle n\rangle.$ Hence the permutation representation is transitive with $\mu(E_4)=|E_4:K|$ for some subgroup $K\le E_4.$ Clearly we must have $\mu(E_4)>p.$ We must show that $\mu(E_4)> p^2.$ Suppose not, so that $\mu(E_4)=p^2.$ Then $\mu(E_4)=|E_4:K|,$ where $K$ is a core-free subgroup of order $p^3$ in $E_4.$ We have from 4.4 in ~\cite{Hall} that $K$ is isomorphic to either one of the following $5$ groups: 
\begin{align*}
&\mathbb{Z}_{p^3},\  \mathbb{Z}_p^2\times\mathbb{Z}_p,\  \mathbb{Z}_p\times\mathbb{Z}_p\times\mathbb{Z}_p,\ H=\langle x,y\; |\; x^{p^2}=y^p=1,\; [x,y]=x^p\rangle\;\text{or}\\
&L=\langle x,y,z\; |\; x^p=y^p=z^p=1,\; [x,y]=z,\; [x,z]=[y,z]=1\rangle. 
\end{align*}
Since $\mu(K)\le\mu(E_4)=p^2,$ it follows that $K$ cannot be isomorphic to $\mathbb{Z}_{p^3}$ or $\mathbb{Z}_p^2\times\mathbb{Z}_p.$ If $K$ is isomorphic to either $H$ or $L$, then since $K$ is a core-free subgroup of $E_4,$ it follows that both $H$ and $L$ intersects trivially with $\langle n\rangle$ so that $H\times\langle n\rangle$ and $L\times\langle n\rangle$ are internal direct products. Now 
$$\mu(H\times\langle n\rangle)=\mu(L\times\langle n\rangle)=p^2+p\le\mu(E_4)=p^2,$$
a contradiction. So we are left with the remaining case that $K\cong\mathbb{Z}_p\times\mathbb{Z}_p\times\mathbb{Z}_p.$ 
%(For your reference David: By someone claiming a result from Burnside's book Theory of groups of finite order: a group of order $p^n$ has an abelian normal subgroup of order $p^m$ where $m(m+1)\ge 2n.$... which I cant seem to find for now in Burnside's book... but I have found something which works just as well from wikipedia with $k(k-1)<2n.$) 
By a result from ~\cite{Burnside}, there exists an abelian normal subgroup $B$ of order $p^3$ in $E_4$. But since $K$ does not contain any non-trivial normal subgroups of $E_4,$ we have $K\cap B=\{1\},$ so we may form the semidirect product $B\rtimes K.$ Moreover, $B\rtimes K$ is a subgroup of $E_4.$ But then $|B\rtimes K|\ge p^6>p^5=|E_4|,$ a contradiction. The proof is now complete.   
%The centre of $E_4$ is of order $p$ (ie generated by the central element, say $n$). Hence the representation is transitive and Johnsons bound, together with the fact that $\mu(G_4)=p^2+p$ forces us to %find a core-free subgroup of index $p^2$ in $E_4.$ I need to show that this is impossible!by
\end{proof}

\end{document}